\documentclass[10pt,reqno]{amsart}
\usepackage{graphicx, amssymb, color,pdfsync}
\usepackage[all,cmtip]{xy}
\numberwithin{equation}{section}
\usepackage{mathrsfs}
\usepackage{hyperref}

\usepackage{lineno}

\hypersetup{
    colorlinks=true,       
    linkcolor=blue,          
    citecolor=blue,        
    filecolor=blue,      
    urlcolor=blue           
}
\usepackage{tikz}

\usetikzlibrary{matrix,arrows}
\DeclareMathOperator{\Aut}{Aut}

\usepackage{yfonts}
\advance\hoffset by -0.9 truecm

\begin{document}
\newcommand{\s}{\vspace{0.2cm}}
\newcommand{\Q}{\mathbb{Q}}
\newcommand{\C}{\mathbb{C}}
\newcommand{\Gal}{\mbox{Gal}}

\newtheorem{theo}{Theorem}
\newtheorem{prop}{Proposition}
\newtheorem{coro}{Corollary}
\newtheorem{lemma}{Lemma}[section]
\newtheorem{example}{Example}
\theoremstyle{remark}
\newtheorem{rema}{\bf Remark}
\newtheorem{defi}{\bf Definition}
\newtheorem*{defisn}{\it Definition}

\title[On families of Riemann surfaces with automorphisms]{On families of Riemann surfaces with automorphisms}
\date{}

\author{Milagros Izquierdo}
\address{Matematiska institutionen, Link\"{o}pings Universitet, Link\"{o}ping, Sweden.}
\email{milagros.izquierdo@liu.se}

\author{Sebasti\'an Reyes-Carocca}
\address{Departamento de Matem\'atica y Estad\'istica, Universidad de La Frontera, Temuco, Chile.}
\email{sebastian.reyes@ufrontera.cl}

\author{Anita M. Rojas}
\address{Departamento de Matem\'aticas, Facultad de Ciencias, Universidad de Chile, \~Nu\~noa, Santiago, Chile.}
\email{anirojas@uchile.cl}

\thanks{The first and second authors were partially supported by Redes Grant 2017-170071. The second author was partially supported by Fondecyt Grants 11180024, 1190991. The third author was partially supported by Fondecyt Grant 1180073.}

\begin{abstract}
In this article we determine the maximal possible order of the automorphism group of the form $ag + b$, where $a$ and $b$ are integers, of a complex three and four-dimensional family of compact   Riemann surfaces of genus $g$, appearing for all genus. In addition, we construct and describe explicit complex three and four-dimensional families possessing these maximal numbers of automorphisms.
\end{abstract}
\dedicatory{Dedicated to our friend Antonio F. Costa on the occasion of his 60th birthday}
\maketitle
\thispagestyle{empty}

\section{Introduction}

The classification of groups of automorphisms of compact Riemann surfaces is a classical subject of study which has attracted broad interest ever since Schwarz and Hurwitz proved that the automorphism group of a compact Riemann surface of genus $g \geqslant 2$ is finite and its order is at most $84g-84.$ 

\s

Riemann surfaces of genus $g$  with a group of automorphisms of order of the form $$ag+b \,\, \mbox{ where } a,b \mbox{ are integers}$$can be found in the literature plentiful supply. The most classical example concerning that is the class of the Riemann surfaces which possess exactly $84g-84$ automorphisms; they are regular covers of the projective line ramified over three values, marked with 2, 3 and 7. Another remarkable example is the cyclic case. Wiman in \cite{Wi} showed that the largest cyclic group of automorphisms of a  Riemann surface of genus $g \geqslant 2$ has order at most $4g+2.$ Furthermore,  the Riemann surface given by \begin{equation*} \label{eWiman}y^2=x^{2g+1}-1\end{equation*}shows that, for each value of $g,$ this upper bound is attained; see also \cite{Harvey1}. Kulkarni in \cite{K1} proved that, for $g$ sufficiently large, the aforementioned curve is the unique Riemann surface of genus $g$ with an automorphism of order $4g+2.$ 

\s

Riemann surfaces with $4g$ automorphisms have been classified in \cite{BCI}; the Jacobian varieties of these surfaces were studied in \cite{yo2}. Riemann surfaces with $8(g+3)$ automorphisms were considered  in \cite{Accola} and  \cite{Mac}. Under the assumption that $g-1$ is a prime number, the case $ag-a$ has been completely classified in \cite{BJ}, \cite{IJRC},  \cite{IRC} and \cite{yo}. Recently, the case $3g-3$ in which $g-1$ is assumed to be the square of a prime number was classified in \cite{CRC}.

\s
 It is worth mentioning that the order of the automorphism group of a  Riemann surface of genus $g$ need not be of the form $ag+b.$ See, for instance,  \cite{CIY2}.

\s

This paper is aimed to  address the problem of determining the maximal possible order of the automorphism group of the form $ag+b,$ where $a$ and $b$ are integers, of a family of   Riemann surfaces of genus $g$, appearing for all genus. To review known facts and  to state the results of this paper, inspired by Accola's notation introduced in \cite{Accola}, we shall bring in the following definition. 

\begin{defisn}For each $d \geqslant 0$ and $A \subseteq \mathbb{N}-\{1\}$ we define $N_d(g, A)$ to be the unique integer of the form $ag+b$ where $a,b \in \mathbb{Z}$, if exists, which satisfies:  \begin{enumerate}
\item for each $g \in A$ there is a complex $d$-dimensional family of  Riemann surfaces of genus $g$ with a group of automorphisms  of order $N_d(g, A),$ and
\item for at least one $g \in A,$ there is no a complex $d$-dimensional family of  Riemann surfaces of genus $g$ with strictly more than $N_d(g, A)$ automorphisms.
\end{enumerate}
If $A=\mathbb{N}-\{1\}$  then we simply write  $N_d(g)$ instead of $N_d(g, A).$
\end{defisn}

In the sixties, Accola \cite{Accola} and Maclachlan \cite{Mac} considered the zero-dimensional case; namely, they dealt with the problem of determining the largest order of the automorphism group of compact Riemann surfaces appearing for all genus. Independently, they proved that $$N_0(g)=8g+8$$by considering the Riemann surface given by the  curve $ \label{eAccola}y^2=x^{2g+2}-1.$ Later, the uniqueness problem was addressed by Kulkarni in \cite{K1}. Concretely, he succeeded in proving that for $g \not\equiv 3 \mbox{ mod } 4$ sufficiently large, the aforementioned curve is the unique Riemann surface of genus $g$ with $8g+8$ automorphisms.

\s

The one-dimensional case was studied in \cite{CI}. For each $g \geqslant 2,$ there is a complex one-dimensional family of  Riemann surfaces of genus $g$ with a group of automorphisms isomorphic to $\mathbf{D}_{g+1} \times C_2$ and $$N_1(g)=4g+4.$$ The uniqueness problem was also studied by noticing that for $g \equiv 3 \mbox{ mod } 4,$ there exists another one-dimensional family with the same number of automorphisms. Besides, the two-dimensional case was addressed  in \cite{yo}, where a classification of compact Riemann surfaces of genus $g$ endowed with a maximal non-large group of automorphisms was studied. By means of this classification, it was noticed that $$N_2(g)=4g-4$$due to the existence, for all $g \geqslant 2,$ of a complex two-dimensional family of  Riemann surfaces of genus $g$ with dihedral action. In addition, it was proved that if $g-1$ is a prime number then the aforementioned family is the unique complex two-dimensional family with this number of automorphisms.

\s

This article is devoted to extend the previous results by dealing with the complex three and four-dimensional cases. Concretely, we first prove that the equality $$N_3(g)=2g-2$$holds. We then observe that this case as well as the zero, one and two-dimensional cases are very much in contrast with the  four-dimensional situation. Indeed, we prove that if $B=\{g : g \geqslant 3\}$ then $$N_4(g,B) \,\, \mbox{does not exist.}$$

In proving the non-existence of $N_4(g,B)$, we obtain the following facts, which are interesting in their own right. If $A_1$ and $A_2$ consist of those values of $g \geqslant 3$ that are odd and even respectively, then $$N_4(g, A_1)=g-1 \,\, \mbox{ and } \,\, N_4(g, A_2)=g.$$

The strategy to prove the results is to find upper bounds for the number of automorphisms and then to construct in a very explicit manner complex three and four-dimensional families attaining these bounds. After that, we study in detail these families; concretely:
\begin{enumerate}
\item we address the uniqueness problem by providing conditions under which they turn into unique,
\item we describe the families themselves as subsets of the moduli space, and
\item we provide an isogeny decomposition of the corresponding families of Jacobian varieties.
\end{enumerate}

Some computations will be done with the help of {\it SageMath} \cite{sage}.
\s

Section \S\ref{s2} is devoted to review  the basic preliminaries. The three-dimensional case will be considered in Sections \S\ref{dim3} and \S\ref{s4}. The four-dimensional case will be considered in Sections \S\ref{dim4}, \S\ref{s6} and \S\ref{s7}.

\section{Preliminaries} \label{s2}
\subsection{Fuchsian groups and group actions} Let $\Delta$ be a Fuchsian group, namely, a discrete group of automorphisms of the upper half-plane $\mathbb{H}.$ If the orbit space $\mathbb{H}_{\Delta}$ given by the action of $\Delta$ on $\mathbb{H}$ is compact, then the algebraic structure of $\Delta$ is determined by its  {\it signature}; namely, the tuple \begin{equation} \label{sig} \sigma(\Delta)=(h; m_1, \ldots, m_l),\end{equation}where $h$ denotes the genus of the quotient surface $\mathbb{H}_{\Delta}$ and $m_1, \ldots, m_l$ the branch indices in the associated universal projection $\mathbb{H} \to \mathbb{H}_{\Delta}.$ If $l=0$ then it is said that $\Delta$ is a surface Fuchsian group.

\s

If $\Delta$ is a Fuchsian group of signature \eqref{sig} then $\Delta$ has a canonical presentation with generators $\alpha_1, \ldots, \alpha_{h}$, $\beta_1, \ldots, \beta_{h},$ $ x_1, \ldots , x_l$ and relations
\begin{equation}\label{prese}x_1^{m_1}=\cdots =x_l^{m_l}=\Pi_{i=1}^{h}[\alpha_i, \beta_i] \Pi_{j=1}^l x_j=1,\end{equation}where $[u,v]$ stands for the commutator $uvu^{-1}v^{-1}.$ The Teichm\"{u}ller space of $\Delta$ is a complex analytic manifold homeomorphic to the complex ball of dimension $3h-3+l$.

\s

Let $\Delta_2$ be a group of automorphisms of $\mathbb{H}.$ If $\Delta$ is a subgroup of $\Delta_2$ of finite index then $\Delta_2$ is also Fuchsian. Moreover, if the signature of $\Delta_2$ is $(h_2; n_1, \ldots, n_s)$ then $$2h-2 + \Sigma_{i=1}^l(1-\tfrac{1}{m_i})= [\Delta_2 : \Delta] (2h_2-2 + \Sigma_{i=1}^s(1-\tfrac{1}{n_i})).$$This equality is called the Riemann-Hurwitz formula.  We refer to \cite{Harvey}, \cite{yoibero} and \cite{singerman} for more details.

\s

Let $S$ be a compact Riemann surface and let $\mbox{Aut}(S)$ denote its automorphism group. A finite group $G$ {\it acts} on $S$ if there is a group monomorphism $G\to \Aut(S).$ The  orbit space $S_G$ is endowed with a Riemann surface structure such that the canonical projection $S \to S_G$ is holomorphic. 

\s

By uniformization theorem, there is a surface Fuchsian group $\Gamma$ such that $S$ and $\mathbb{H}_{\Gamma}$ are isomorphic. Moreover, Riemann's existence theorem ensures that $G$ acts on $S \cong \mathbb{H}_{\Gamma}$ if and only if there is a Fuchsian group $\Delta$ containing $\Gamma$ together with a group  epimorphism \begin{equation*}\label{epi}\theta: \Delta \to G \, \, \mbox{ such that }  \, \, \mbox{ker}(\theta)=\Gamma.\end{equation*}

Note that $S_G \cong \mathbb{H}_{\Delta}.$ It is said that $G$ acts on $S$ with signature $\sigma(\Delta)$ and that this action is {\it  represented} by the {\it surface-kernel epimorphism} $\theta$. For the sake of simplicity, we usually identify $\theta$ with the tuple of its images or {\it generating vector:} (see, for example, \cite{Brou} and \cite{yoibero})$$\theta = (\theta(\alpha_1), \ldots, \theta(\alpha_h),\theta(\beta_1), \ldots, \theta(\beta_h), \theta(x_1), \ldots, \theta(x_l)).$$

%

\subsection{Equivalence of actions} \label{braid}

Let $S_1$ and $S_2$ be two compact Riemann surfaces of the same genus. Two actions $\psi_i: G \to \mbox{Aut}(S_i)$, for $i=1,2,$ are {\it topologically equivalent} if there exist $\omega \in \mbox{Aut}(G)$ and an orientation-preserving homeomorphism $f : S_1 \to S_2$ such that
\begin{equation}\label{equivalentactions}
\psi_2(g) = f \psi_1(\omega(g)) f^{-1} \mbox{ for all } g \in G.
\end{equation}Observe that if we write $(S_i)_{G} \cong \mathbb{H}_{\Delta_i}$ then each homeomorphism $f$ satisfying (\ref{equivalentactions}) induces  an isomorphism $f^*: \Delta_1 \to \Delta_2$. Thus, in order to describe the topological equivalence classes of a given group, one may assume in the above that $\Delta_1=\Delta_2.$ We shall write $\Delta$ instead of $\Delta_i$ and $S$ instead of $S_i.$

We denote the subgroup of $\mbox{Aut}(\Delta)$ consisting of those $f^*$ by $\mathfrak{B}$. It is known that $\theta_1, \theta_2 : \Delta \to G$ define topologically equivalent actions if and only if there are $\omega \in \mbox{Aut}(G)$ and $f^* \in \mathfrak{B}$ such that  $\theta_2 = \omega\circ\theta_1 \circ f^*$ (see \cite{Brou} and \cite{Harvey}). We recall for later use that, with the notations \eqref{prese}, if the genus $h$ of $S_G$ is zero,  then $\mathfrak{B}$ is generated by the  {\it braid transformations} $\Phi_{i}$ defined by: \begin{equation} \label{bd}\Phi_i: x_i \mapsto x_{i+1}, \hspace{0.3 cm}x_{i+1} \mapsto x_{i+1}^{-1}x_{i}x_{i+1} \hspace{0.3 cm} \mbox{ and }\hspace{0.3 cm} x_j \mapsto x_j \mbox{ when }j \neq i, i+1\end{equation}
for each $i \in \{1, \ldots, l-1\}.$ Meanwhile, if $h=1$, then, in addition to \eqref{bd}, $\mathfrak{B}$ contains $$A_{1,n}: \alpha_1 \mapsto \alpha_1, \,\, \beta_1 \mapsto \beta_1 \alpha_1^n, \,\,  x_j \to x_j, \hspace{ 0,5 cm } A_{2,n}: \alpha_1 \mapsto \alpha_1 \beta_1^n, \,\, \beta_1 \mapsto \beta_1, \,\,  x_j \to x_j$$where $n \in \mathbb{Z},$ and the transformations $$ C_{1,i}: \alpha_1 \mapsto x_1 \alpha_1, \,\, \beta_1 \mapsto \beta_1, \,\, x_i \mapsto y_1x_iy_1^{-1}, \,\, x_j \mapsto x_j \,\, \mbox{ for each }\, j \neq i$$ $$ C_{2,i}: \alpha_1 \mapsto   \alpha_1, \,\, \beta_1 \mapsto x_2\beta_1, \,\, x_i \mapsto y_2x_iy_2^{-1}, \,\, x_j \mapsto x_j \,\, \mbox{ for each }\, j \neq i $$for $i \in \{1, \ldots, l\}$, where $x_1=\beta_1^{-1}wz,  y_1=z\beta_1^{-1}w,  x_2=wz\alpha_1, y_2=z\alpha_1 w$, $w=\Pi_{k < i} x_k$ and $z=\Pi_{k > i} x_k.$ See, for example, \cite{bcitesis}, \cite{Brou} and \cite{Harvey}.

\s

Let $\mathscr{B}_g$ denote the locus of orbifold-singular points of the moduli space $\mathscr{M}_g$ of  Riemann surfaces of genus $g.$ It was proved in \cite{b} (see also \cite{Harvey}) that $\mathscr{B}_g$ admits an {\it equisymmetric stratification}, \begin{equation*}\label{stratif} \mathscr{B}_g = \cup_{G, \theta} \bar{\mathscr{M}}_g^{G, \theta}\end{equation*}where the non-empty {\it equisymmetric strata} are in bijective correspondence with the topological classes of  actions that are maximal (in the sense  of \cite{singerman2}). Concretely:

\begin{enumerate}

\item the {\it equisymmetric stratum} ${\mathscr{M}}_g^{G, \theta}$ consists of  those Riemann surfaces $S$ of genus $g$ with (full) automorphism group isomorphic to $G$ such that the action is topologically equivalent to $\theta$,

\item the closure $\bar{\mathscr{M}}_g^{G, \theta}$ of  ${\mathscr{M}}_g^{G, \theta}$ is a closed irreducible algebraic subvariety of $\mathscr{M}_g$ and consists of those Riemann surfaces $S$ of genus $g$ with a group of automorphisms isomorphic to $G$ such that the action is  topologically equivalent to $\theta$, and

\item  if the equisymmetric stratum ${\mathscr{M}}_g^{G, \theta}$ is non-empty then it is a smooth, connected,
locally closed algebraic subvariety of $\mathscr{M}_{g}$ which is Zariski dense in
$\bar{\mathscr{M}}_g^{G, \theta}.$ 
\end{enumerate}

\begin{defisn} Let $G$ be a group and let $\sigma$ be a signature. The subset of $\mathscr{M}_g$ consisting of all those  Riemann surfaces of genus $g$ endowed with a group of automorphisms isomorphic to $G$ acting with signature $\sigma$ will be called a {\it closed family} or simply a {\it family}.
\end{defisn}

We recall that:
\begin{enumerate}
\item the complex dimension of the family agrees with the complex dimension of the Teichm\"{u}ller space associated to a Fuchsian group of signature $\sigma,$
\item the interior of a family, if non-empty, consists of those Riemann surfaces whose (full) automorphism group is isomorphic to $G$ and is formed by finitely many equisymmetric strata which are in correspondence with the pairwise non-equivalent topological actions of $G,$ and 
\item the complement of the interior (with respect to the family) is formed by those Riemann surfaces that have strictly more automorphisms than $G.$ 
\end{enumerate}

\begin{defisn} A  family is called {\it equisymmetric} if its interior consists of exactly one equisymmetric stratum.
\end{defisn}

\subsection{Jacobian and Prym varieties} \label{jacos}  Let $S$ be a compact Riemann surface of genus $g \geqslant 2.$ We denote by $\mathscr{H}^1(S, \mathbb{C})^*$ the dual of the $g$-dimensional complex vector space of  1-forms on $S,$ and by $H_1(S, \mathbb{Z})$ the first integral homology group of $S.$
We recall that the {\it Jacobian variety} of $S,$ defined by $$JS=\mathscr{H}^1(S, \mathbb{C})^*/H_1(S, \mathbb{Z}),$$is an irreducible principally polarized abelian variety of dimension $g.$ The relevance of  the Jacobian variety  lies, partially, in  Torelli's  theorem, which establishes that two Riemann surfaces are isomorphic if and only if the corresponding Jacobian varieties are isomorphic as principally polarized abelian varieties. 

\s

If $H \leqslant \mbox{Aut}(S)$ then the  associated regular covering map $\pi : S \to S_H$   induces a homomorphism $$\pi^*: JS_H \to JS$$between the associated Jacobians. The image of $\pi^*$ is an abelian subvariety of $JS$ isogenous to $JS_H$. Thereby,    the classical Poincar\'e's Reducibility theorem implies that there exists an abelian subvariety  of  $JS,$ henceforth  denoted by $\mbox{Prym}(S \to S_H)$ and called the {\it Prym variety} associated to $\pi,$ such that  $$JS \sim JS_H \times \mbox{Prym}(S \to S_H),$$where $\sim$ stands for isogeny. See \cite{bl} for more details.

\s

If $G$ acts on a compact Riemann surface $S$ then this action  induces a $\mathbb{Q}$-algebra homomorphism $$\rho : \mathbb{Q} [G] \to \mbox{End}_{\mathbb{Q}}(JS)=\mbox{End}(JS) \otimes_{\mathbb{Z}} \mathbb{Q},$$from the rational group algebra of $G$  to the rational endomorphism algebra of $JS.$ For  $ \alpha \in {\mathbb Q}[G]$, set  $$A_{\alpha} := {\textup Im} (\alpha)=\rho (n\alpha)(JS) \subseteq JS$$where $n$ is a suitable positive integer chosen in such a way that $n\alpha \in {\mathbb Z}[G]$.

Let $W_1, \ldots, W_r$ be the rational irreducible representations of $G$ and for each $W_l$ denote by $V_l$ a complex irreducible representation of $G$ associated to it. Following \cite{l-r}, the decomposition  of $1$ as the sum $e_1 + \cdots + e_r$ in $\mathbb{Q}[G],$ where $e_l$ is central idempotent associated to $W_l,$ yields a $G$-equivariant isogeny $$JS \sim A_{e_1} \times \cdots \times A_{e_r}.$$
 Moreover, for each $l$ there are idempotents $f_{l1},\dots, f_{ln_l}$ such that $e_l=f_{l1}+\dots +f_{ln_l}$ where  $n_l=d_{V_l}/s_{V_l}$ is the quotient of the degree $d_{V_l}$ of $V_l$ and its Schur index $s_{V_l}$. These idempotents provide $n_l$ pairwise isogenous subvarieties of $JS;$ let $B_l$ be one of them, for each $l.$ Thus, the following isogeny is obtained
\begin{equation} \label{eq:gadec}
JS \sim B_{1}^{n_1} \times \cdots \times B_{r}^{n_r} 
\end{equation}
and is called the {\it group algebra decomposition} of $JS$ with respect to $G$. See \cite{cr} and also \cite{RCR}.

If $W_1$ denotes the trivial representation then $n_1=1$ and $B_{1} \sim JS_G$.

\s

If $H \leqslant G$ then we  denote by $d_{V_l}^H$  the dimension of the vector subspace $V_l^H$ of $V_l$  of the elements fixed under $H.$ Following \cite{cr},   the group algebra decomposition   \eqref{eq:gadec} induces the following isogenies.

\begin{enumerate}
\item The Jacobian variety $JS_H$ of the quotient $S_H$ decomposes as\begin{equation} \label{decoind1}
JS_H \sim  B_{1}^{{n}_1^H} \times \cdots \times B_{r}^{n_r^H} \,\,\, \mbox{ where } \,\,\, {n}_l^H=d_{V_l}^H/s_{V_l}.
\end{equation}
\item Let $H_1 \leqslant H_2$ be subgroups of $G.$ The    Prym variety associated to  $S_{H_1} \to S_{H_2}$ decomposes as \begin{equation} 
\label{decoind2}
\mbox{Prym}(S_{H_1} \to S_{H_2}) \sim  B_{1}^{{n}_1^{H_1, H_2}} \times \cdots \times B_{r}^{{n}_r^{H_1, H_2}} \,\,\, \mbox{ where } \,\,\, {n}_l^ {H_1, H_2}=n_l^{H_1}-n_l^{H_2}.
\end{equation}
\end{enumerate}

The previous induced isogenies have been useful to provide decomposition of Jacobian varieties $JS$ whose factors are isogenous to Jacobians of quotients of $S$ and Pryms of intermediate coverings; see, for example, \cite{CMA}, \cite{CRC} and \cite{kanirubiyo}. 

Assume that $(\gamma; m_1, \ldots, m_l)$ is the signature of the action of $G$ on $S$  and that this action is represented by the surface-kernel epimorphism $\theta: \Delta \to G,$ with $\Delta$ canonically presented as in \eqref{prese}. As proved  in \cite[Theorem 5.12]{yoibero},  the dimension of $B_{i}$ in \eqref{eq:gadec} for  $i \geqslant 2$ is given by 
\begin{equation}\label{dimensiones1}
\dim B_{i}=k_{V_i}[d_{V_i}(\gamma -1)+\frac{1}{2}\Sigma_{k=1}^l (d_{V_i}-d_{V_i}^{\langle \theta(x_k) \rangle} )]  \end{equation} where $k_{V_i}$ is the degree of the extension $\mathbb{Q} \le L_{V_i}$ with $L_{V_i}$ denoting a minimal field of definition for $V_i.$ Note that the dimension of $B_1$ equals $\gamma.$

\s

The decomposition of Jacobian varieties with group actions goes back to old works of Wirtinger, Schottky and Jung. For decompositions of Jacobians with respect to special groups, we refer to the articles \cite{Ba}, \cite{d1}, \cite{Do}, \cite{nos}, \cite{IJR}, \cite{LR2}, \cite{PA}, \cite{d3} and \cite{Ri}.

\subsection*{\it Notation} We  denote the cyclic group of order $n$ by $C_n$ and the dihedral group of order $2n$ by $\mathbf{D}_n.$

\section{The three-dimensional case} \label{dim3}

\begin{theo} 
$N_3(g)=2g-2.$
\end{theo}

The proof of the theorem will follow directly from Lemmata \ref{L11} and \ref{L12} stated and proved below.
 
\begin{lemma} \label{L11} Let $g \geqslant 2$ be an integer. There are no complex three-dimensional families of compact Riemann surfaces of genus $g$ with strictly more than $2(g-1)$ automorphisms.
\end{lemma}

\begin{proof}
Assume the existence of a complex three-dimensional family of  Riemann surfaces $S$ of genus $g$ with a group of automorphisms $G$ of order strictly greater that $2(g-1).$ If the signature of the action of $G$ on $S$ is $(h; m_1, \ldots, m_l)$ then, by the Riemann-Hurwitz formula, we have that $$2(g-1) > 2(g-1)[2h-2+\Sigma_{j=1}^l(1- \tfrac{1}{m_j})],$$or, equivalently,  $\label{m1}\Sigma_{j=1}^l \tfrac{1}{m_j} > 2h+l-3.$ As the dimension $3h-3+l$ of the family  is assumed to be 3, \begin{equation}\label{m1}\Sigma_{j=1}^l \tfrac{1}{m_j} > 1+ \tfrac{l}{3} \,\, \mbox{ where }\,\, l\in \{0, 3, 6\}.\end{equation}If $l=0$ then \eqref{m1} turns into $0  > 1.$ Besides, if $l=3$ or $l=6$ then \eqref{m1} turns into $\Sigma_{j=1}^3\tfrac{1}{m_j} > 2$ and $\Sigma_{j=1}^6\tfrac{1}{m_j} > 3$ respectively. In both cases this contradicts the fact that each $m_j$ is at least 2.
\end{proof}

\begin{lemma} \label{L12} Let $g \geqslant 2$ be an integer. There is a complex three-dimensional family  of compact Riemann surfaces $S$ of genus $g$ with a group of automorphisms  $G$ isomorphic to the dihedral group of order $2(g-1)$ such that the signature of the action of $G$ on $S$ is $(0; 2, \stackrel{6}{\ldots}, 2).$
\end{lemma}

\begin{proof}
Let $\Delta$ be a Fuchsian group of signature $(0; 2, \stackrel{6}{\ldots}, 2)$ with canonical presentation \begin{equation*} \label{cenn}\Delta=\langle x_1, \ldots, x_6 : x_1^2 = \cdots = x_6^2=x_1\cdots x_6=1\rangle,\end{equation*}and consider the dihedral group $\mathbf{D}_{g-1}=\langle r, s : r^{g-1}=s^2=(sr)^2=1\rangle.$ Note that if $g \geqslant 3$ then  \begin{equation} \label{action1}\Delta \to \mathbf{D}_{g-1} \, \mbox{ given by } \, x_1, \ldots, x_4 \mapsto s \, \mbox{ and } \, x_5, x_6 \mapsto sr\end{equation}is a surface-kernel epimorphism of signature $(0; 2, \stackrel{6}{\ldots}, 2)$. If $g=2$ then the group is $C_2= \langle s \rangle$ and the surface-kernel epimorphism can be chosen to be $x_j \mapsto s$ for each $1 \leqslant j \leqslant 6.$ In addition, for each $g \geqslant 2,$ the equality $$2(g-1)=2(g-1)[0 - 2 + 6(1-\tfrac{1}{2} )]$$shows that the Riemann-Hurwitz formula is satisfied for a $2(g-1)$-fold regular covering map from a   Riemann surface of genus $g$ onto the projective line with six branch values marked with 2. 

Thus, the existence of the desired family  follows from Riemann's existence theorem.
\end{proof}

{\bf Notation.} From now on, we shall denote  the family  of all those surfaces $S$ of genus $g \geqslant 2$ with a group of automorphisms  $G$ isomorphic to the dihedral group of order $2(g-1)$ such that the signature of the action of $G$ on $S$ is $(0; 2, \stackrel{6}{\ldots}, 2)$ by $\mathcal{F}_g$.

%

\section{The family $\mathcal{F}_g$} \label{s4}

\begin{prop} \label{L13} Let $g \geqslant 3$ be an integer. 
If $g-1$ is a prime number then $\mathcal{F}_g$ is the unique complex three-dimensional family of compact Riemann surfaces of genus $g$ with $2(g-1)$ automorphisms.
\end{prop}

\begin{proof} Set $q=g-1.$ Let $\mathcal{F}$ be a complex three-dimensional family of  Riemann surfaces of genus $g$ with a group of  automorphisms $G$ of order $2q.$ By considering the Riemann-Hurwitz formula and by arguing similarly as done in the proof of Lemma \ref{L11}, one sees that the unique solution of $$1=2h-2+\Sigma_{j=1}^l(1-\tfrac{1}{m_j})$$is $h=0, l=6$ and $m_j=2$ for each $1 \leqslant j \leqslant 6$. Thus, the signature of the action of $G$ on each $S \in \mathcal{F}$ is necessarily equal to $(0; 2, \stackrel{6}{\ldots}, 2).$ If we now assume $q$ to be prime then $G$ is isomorphic to either the dihedral group or the cyclic group. We claim that the latter case is impossible. In fact, otherwise 
there would exist a surface-kernel epimorphism $\Delta \to C_{2q}$ where $\Delta$ is a Fuchsian group of signature $(0; 2, \stackrel{6}{\ldots}, 2).$ This, in turn, would imply that $C_{2q}$ can be generated by involutions; a contradiction. It follows that $G$ is isomorphic to the dihedral group and therefore $\mathcal{F}$ agrees with the family $\mathcal{F}_g$ as desired.
\end{proof}

\begin{prop}\label{L14} Let $g \geqslant 4$ be an integer. If $g-1$ is a prime number  then $\mathcal{F}_g$ is equisymmetric.
\end{prop}

\begin{proof} Set $q=g-1$ and assume $q$ to be prime. Let $\theta: \Delta \to \mathbf{D}_{q}=\langle r, s : r^q=s^2=(sr)^2=1\rangle$ be a surface-kernel epimorphism representing an action of $G$ on $S \in \mathcal{F}_g.$ What we need to prove is that $\theta$ is equivalent to the surface-kernel epimorphism \eqref{action1} of Lemma \ref{L12}. To accomplish this task we shall introduce some notation. We write $$sr^{n_j}=\theta(x_j) \, \mbox{ where } \,  j=1, \ldots, 6 \, \mbox{ and } \, n_j \in \{0, \ldots, q-1\},$$and if $n_j \neq 0$ then we shall denote by $m_j$ its  inverse in the field of $q$ elements. Also, we denote by $\phi_{\alpha,\beta}$ the automorphism of $\mathbf{D}_q$ given by $(r,s) \mapsto (r^{\alpha}, sr^{\beta})$ for $1 \leqslant \alpha \leqslant q-1$ and $0 \leqslant \beta \leqslant q-1.$

\s

{\bf Claim 1.} Let $j \in \{1, \ldots, 5\}$ fixed. If $n_j=1$ and $n_k=0$ for all  $k < j$ then, up to equivalence, we can assume that $n_{j+1}=0$ or $n_{j+1}=1.$

\s

Assume $n_{j+1} \neq 0.$  Then the transformation $\phi_{m_{j+1},0} \circ \Phi_{j}$ induces the correspondence $$(s, \stackrel{j-1}{\ldots}, s,sr,sr^{n_{j+1}}) \to (s, \stackrel{j-1}{\ldots}, s,sr,sr^{f(n_{j+1})}) \,\, \mbox{ where } f(u)=2-\tfrac{1}{u}.$$ The claim follows by noting that the rule $u \mapsto f(u) $ fixes 1 and has an orbit of length is $q-1.$ 


\s

{\bf Claim 2.} Up to equivalence, we can assume $n_1=n_2=0.$

\s

Note that if $n_1=n_2$ then it is enough to consider  $\phi_{1, -n_1}$ to obtain the claim. Thus, we shall assume that $n_1 \neq n_2.$  If $\alpha:=(n_2-n_1)^{-1}$ and $\beta:=n_1(n_1-n_2)^{-1}$ (where the inverses are taken in the field of $q$ elements) then  the automorphism $\phi_{\alpha, \beta}$ ensures that, up to equivalence, $n_1=0$ and $n_2=1.$ Now:

\begin{enumerate}
\item[(a)] if $n_3=0$ then $\Phi_{2}$ shows that we can assume $n_1=n_2=0,$ and 
\item[(b)] if $n_3 \neq 0$ then, by Claim 1, we can assume $n_3=1.$ We now apply $\Phi_2 \circ \Phi_1 \circ \phi_{-1,1}$ to obtain that, up to equivalence,  $n_1=n_2=0.$ 
\end{enumerate}

The proof of the claim is done.

\s

We proceed by studying two cases separately, according to $n_3=0$ or $n_3 \neq 0.$

\s

{\it Type 1.} Assume that $n_3=0.$ 
\begin{enumerate}
\item[(a)] If $n_4=0$ then necessarily $n_5$ and $n_6$ are equal and different from zero. We consider $\phi_{m_5, 0}$ to obtain that $\theta$ is equivalent to \eqref{action1}. 
\item[(b)] If $n_4 \neq 0$ then, we consider $\phi_{m_4, 0}$ to assume $n_4=1.$ Now, by Claim 1, we can ensure that $n_5=0$ or $n_6=1;$ thus, $\theta$ is equivalent to either \begin{equation} \label{case1}\theta_1=(s,s,s,sr,s,sr^{-1}) \,\, \mbox{ or }\,\, \theta_2=(s,s,s,sr,sr,s).\end{equation}Note that $\theta_1$ and $\theta_2$ are equivalent under $\Phi_5$ and that, in turn, $\theta_1$  is equivalent to \eqref{action1} under the action of $\phi_{-1,0} \circ \Phi_3.$
\end{enumerate}

\s

{\it Type 2.} Assume that $n_3 \neq 0.$ As before, by considering the automorphism $\phi_{m_3, 0},$ we can assume $n_3=1.$ It follows, by Claim 1, that $n_4=0$
 or $n_4=1.$ The first case can be disregarded, since $\phi_{-1,0} \circ \Phi_3$ provides an equivalence with \eqref{case1}. Now, if $n_4=1$ then $\theta$ is equivalent to $$\theta_u=(s,s,sr,sr,sr^u, sr^{u}) \,\, \mbox{ for some }\,\, u \in \{0, \ldots, q-1\}.$$

 \begin{enumerate}
\item if $u \neq \pm 1$ then define $\alpha_u$ and $\beta_u$ by $\alpha_u(1-u) \equiv 1 \mbox{ mod }q$ and $\beta_u(1+u) \equiv 1 \mbox{ mod }q.$ The transformation $ \phi_{\beta_u, 0} \circ \Phi_{4}^{\beta_u} \circ \Phi_{5} \circ \Phi_{3} \circ \Phi_{4}^{\alpha_u}$ shows that $\theta_u$ is equivalent to \eqref{action1}.
\s
\item if $u=1$ or $u=-1$ then we consider the transformations $$\Phi_{4} \circ \Phi_{3} \circ\Phi_{2} \circ\Phi_{1} \circ \Phi_{5} \circ \Phi_{4} \circ \Phi_{3} \circ\Phi_{2} \circ\phi_{-1, 1} \,\, \mbox{ and }\,\, \Phi_{4} \circ\Phi_{5}^2 \circ\Phi_{3} \circ \Phi_{4}^{\alpha}$$ respectively (where $2\alpha=1$) to see that $\theta_u$ is equivalent to \eqref{action1}.
\end{enumerate}The proof of the proposition is done.
\end{proof}

\s

We shall denote the equisymmetric stratum corresponding to the action \eqref{action1}  by  $\mathcal{F}_{g,1}$. With this terminology the previous proposition can be rephrased as  $$g-1 \, \mbox{ odd prime } \implies \mathcal{F}_{g,1}=\mathcal{F}_g.$$

In order to state the following proposition we need some notation. For each integer $n \geqslant 2$ we write $$\Omega(n)=\{d \in \mathbb{Z}: d \text{ divides } n \text{ and } 1 \leqslant d < n\}$$and for each $n \geqslant 2$ even we write $$\hat{\Omega}(n)=\{d \in \mathbb{Z}: d \text{ divides } n \text{ and } 1 \leqslant d < \tfrac{n}{2}\}.$$ Let $\varphi$ denote the Euler function.

%

\begin{prop}\label{desco3dim} Let $g \geqslant 4$ be an integer. 
\begin{enumerate}
\item[(a)] If $g-1$ is odd and $S \in \mathcal{F}_g$ then $JS$  decomposes, up to isogeny, as\begin{equation*}\label{blanco1}JS \sim A \times \Pi_{d\in \Omega(g-1)}B_d^2\end{equation*}
where $A$ is an abelian surface and $B_d$ is an abelian variety of dimension $\tfrac{1}{2}\varphi(\tfrac{g-1}{d}).$ Moreover $$JS_{\langle r \rangle} \sim A \,\, \mbox{ and } \,\, JS_{\langle s \rangle} \sim   \Pi_{d \in \Omega(g-1)}B_d$$and therefore $JS \sim JS_{\langle r \rangle}  \times JS_{\langle s \rangle}^2.$

\s

\item[(b)] If $g-1$ is even and $S \in \mathcal{F}_{g,1}$ then $JS$  decomposes, up to isogeny, as\begin{equation*}\label{blanco2}JS \sim E \times A  \times \Pi_{d \in \hat{\Omega}(g-1)}B_d^2,\end{equation*}
where $A$ is an abelian surface, $B_d$ is an abelian variety of dimension $\tfrac{1}{2}\varphi(\tfrac{g-1}{d})$ and $E$ is an elliptic curve. Moreover, $$JS_{\langle r \rangle} \sim A, \,\,  JS_{\langle s \rangle} \sim   \Pi_{d \in \hat{\Omega}(g-1)}B_d \,\, \mbox{ and } \,\, JS_{\langle sr \rangle} \sim E \times  \Pi_{d \in \hat{\Omega}(g-1)}B_d$$and therefore $JS \sim JS_{\langle r \rangle}  \times JS_{\langle s \rangle} \times JS_{\langle sr \rangle}.$

\end{enumerate}
\end{prop}

\begin{proof} We write $n:=g-1.$  We assume that $n$ is odd. It is well-known that the complex irreducible representations of $\mathbf{D}_{n}= \langle r, s : r^{n}=s^2=(sr)^2=1 \rangle$ are, up to equivalence:
 \begin{enumerate}
\item two of degree $1$: the trivial representation denoted by $\chi_1$ and $\chi_2 : r \mapsto 1, \,\, s \mapsto -1.$
\item $\frac{n-1}{2}$ of degree $2$, given by$$\psi_j: r \mapsto  \mbox{diag}(\omega^j, \bar{\omega}^j)
 \,\, \mbox{ and } \,\,  s \mapsto  \left( \begin{smallmatrix}
0 & 1  \\
1 & 0  \\
\end{smallmatrix} \right),$$where $\omega$ is a primitive $n$-root of unity and 
$j=1,\ldots, \tfrac{n-1}{2}.$
\end{enumerate}
\s
For $d \in \Omega(n),$ we denote by $K_d$ the character field of $\psi_d$ (an extension of $\mathbb{Q}$ of degree $\frac{1}{2}\varphi(\frac{n}{d})$) and define \begin{equation} \label{repreW}W_d := \oplus_{\sigma \in G_d}\psi_d^{\sigma},\end{equation}where $G_d$ stands for the Galois group associated to $\mathbb{Q} \leqslant K_d.$ Following for example \cite[Section 2]{IJR}, up to equivalence,  the rational  irreducible representations of $\mathbf{D}_n$ are $ \chi_1,  \chi_2$ and $W_d$ with $d \in \Omega(n).$

We recall that the Schur index of each representation of a dihedral group equals 1. Thus,  if $S \in \mathcal{F}_g$ then the group algebra decomposition of $JS$ with respect to $G$ is \begin{equation} \label{negro1}JS \sim B_2 \times \Pi_{d \in \Omega(n)}B_d^2,\end{equation}where  the factor $B_1$ is disregarded since the genus of $S_G$ is zero. 

Note that as $n$ is assumed to be odd, all the involutions of $\mathbf{D}_n$ are pairwise conjugate and therefore the dimension of the corresponding fixed subspaces agree. This simple fact implies that the dimension of each factor in \eqref{negro1} does not depend on the equisymmetric stratum to which $S$ belongs.
Then, in order to apply the formula \eqref{dimensiones1} we only need to compute the dimension of the fixed subspaces of $\chi_2$ and $\psi_d$ under the action of  $\langle s \rangle.$ In the former case we have that $$\chi_2^{\langle s \rangle}=0 \,\, \mbox{ and therefore} \,\, \dim B_2=-1+\tfrac{1}{2}(6(1-0))=2,$$meanwhile in the latter case, for each $d \in \Omega(n),$ we have $$\psi_d^{\langle s \rangle}=1\,\, \mbox{ and therefore} \,\,\dim B_d=\tfrac{1}{2}\varphi(\tfrac{n}{d})(-2+\tfrac{1}{2}(6(2-1))=\tfrac{1}{2}\varphi(\tfrac{n}{d}).$$

Finally, we apply the induced isogeny \eqref{decoind1} with $H= \langle r \rangle$ and $H=\langle s \rangle$ to obtain that $JS_{\langle r \rangle} \sim B_2$ and $JS_{\langle s \rangle} \sim   \Pi_{d \in \Omega(n)}B_d$ respectively. The proof of the statement (a) follows after setting $A=B_2.$

\s

We now assume that $n$ is  even and proceed analogously. The complex irreducible representations of $\mathbf{D}_{n}$ are, up to equivalence,  the trivial one $\chi_1$,   $$\chi_2 : r \mapsto 1, \,\, s \mapsto -1, \,\, \chi_3 : r \mapsto -1, s \mapsto 1 \,\mbox{ and }\, \chi_4 : r \mapsto -1, s \mapsto -1. $$and $\frac{n}{2}-1$ of degree $2$, given by $\psi_j$ with 
$j=1,\ldots, \tfrac{n}{2}-1.$

\s

Up to equivalence, the rational  irreducible representations of $\mathbf{D}_n$ are $ \chi_1,  \chi_2, \chi_3, \chi_4$ and $W_d$ with $d \in \hat{\Omega}(n).$  If $S \in \mathcal{F}_{g,1}$ then the  group algebra decomposition of $JS$ with respect to $G$ is  \begin{equation*} \label{negro2}JS \sim B_2 \times B_3 \times B_4 \times \Pi_{d \in \hat{\Omega}(n)}B_d^2,\end{equation*}where, as before, $B_1$ is not considered. Note that $$\chi_2^{\langle s \rangle}=\chi_2^{\langle sr \rangle}=0, \,\, \chi_3^{\langle s \rangle}=1, \, \chi_3^{\langle sr \rangle}=0 \,\, \mbox{ and } \,\, \chi_4^{\langle s \rangle}=0, \, \chi_4^{\langle sr \rangle}=1$$and for each $d \in \hat{\Omega}(n)$ we have that $ \psi_d^{\langle s \rangle}= \psi_d^{\langle sr \rangle}=1.$ Then, we apply \eqref{dimensiones1} to conclude that $$\dim B_2 =2, \,\, \dim B_3 =0, \,\, \dim B_4=1 \, \mbox{ and }\,\dim B_d= \tfrac{1}{2}\varphi(\tfrac{n}{d}).$$

Finally, we consider the induced isogeny \eqref{decoind1} with $H= \langle r \rangle$, $H=\langle s \rangle$ and $H=\langle sr \rangle$ to obtain that  $JS_{\langle r \rangle} \sim B_2, JS_{\langle s \rangle} \sim   \Pi_{d \in \hat{\Omega}(n)}B_d$ and $JS_{\langle sr \rangle} \sim B_4 \times  \Pi_{d \in \hat{\Omega}(n)}B_d$ respectively. The proof of the statement (b) follows after setting $E=B_4$ and $A=B_2.$
\end{proof}

\begin{rema}
We end this section by pointing out some remarks concerning the family $\mathcal{F}_g.$

\begin{enumerate}
\item Note that if $g-1$ is an odd prime (or, more generally, odd) then $\mathbf{D}_{g-1}$ does not contain central subgroups of order two. Thus, generically, each $S \in \mathcal{F}_g$ is non-hyperelliptic.
\s

\item If $g-1$ is prime then $\mathcal{F}_g$ corresponds to the family of  Riemann surfaces of genus $g$ that are cyclic unbranched covers of  Riemann surfaces of genus two.
\s
\item The family $\mathcal{F}_3$ consists of two equisymmetric strata: one of them represented by \eqref{action1} and the other represented by $(s,s,r,r,sr,sr)$. See \cite[Table 5, 3.h]{Brou}.

\s

\item If $g-1$ is not prime then Proposition \ref{L14} is not longer true. For instance, if $g-1$ is even, then $\theta_c:=(r^{(g-1)/2}, r^{(g-1)/2}, s,s,sr,sr)$ defines an action which is non-equivalent to \eqref{action1}.
\s
\item  For each $g \geqslant 2,$ the (closed) family $\mathcal{F}_g$ contains the complex two-dimensional family with the maximal possible number of automorphisms (see \cite{yo}), which does not lie in the interior of  $\mathcal{F}_g$ (see  Subsection \S\ref{braid} for the definition of {\it interior}; see also \cite{b}).

\s

\item The group algebra decomposition of Jacobians  of Riemann surfaces which belong to the same family but lying in different  strata may differ radically. For instance, if $g \equiv 3 \mbox{ mod } 4$ and $S$ belongs to the stratum defined by $\theta_c$, then the group algebra decomposition of $JS$ has three factors  of dimension one, instead of only one as in the stratum \eqref{action1}.
\s
\item Note that in Proposition \ref{desco3dim}(a) the Jacobian varieties  $JS_{\langle s \rangle}$ and $JS_{\langle sr \rangle}$ are isomorphic. Thus, independently of the parity of $g,$ if $S \in \mathcal{F}_{g,1}$ then $JS$ is isogenous to $JS_{\langle r \rangle}  \times JS_{\langle s \rangle} \times JS_{\langle sr \rangle}.$ 

\s

\item Kani and Rosen in \cite{KR} provided conditions under which the Jacobian of a Riemann surface $S$ decomposes, up to isogeny, as a product of Jacobians of quotients of $S$. In spite of the fact that $JS \sim JS_{\langle r \rangle}  \times JS_{\langle s \rangle} \times JS_{\langle sr \rangle}$ for each $S$ as in Proposition \ref{desco3dim}, the previous isogeny cannot be derived from Kani-Rosen's result (the reason is that $\langle s \rangle$ and $\langle sr \rangle$ do not permute). It is worth mentioning that the  isogenies of Proposition \ref{desco3dim} (and the ones of Proposition \ref{celular} stated later) can be also obtained by applying the main result of \cite{kanirubiyo}.

\s

\item In \cite{yonil} it was proved that the maximal order of a nilpotent group of automorphisms of a  three-dimensional family of  Riemann surfaces of genus $g$ is $2(g-1).$ If $g-1$ is a power of 2 then  $\mathbf{D}_{g-1}$ is nilpotent, showing that  $\mathcal{F}_g$ attains the aforesaid upper bound for infinitely many values of $g.$
\end{enumerate}
\end{rema}

\section{The four-dimensional case} \label{dim4}

\begin{lemma} \label{L21} Let $g \geqslant 4$ be an even integer. If $g-1$ is a prime number then 
there are no complex four-dimensional families of compact Riemann surfaces of genus $g$ with strictly more than $g$ automorphisms.\end{lemma}

\begin{proof}
Assume the existence of a complex four-dimensional family of  Riemann surfaces of genus $g$ with a group of automorphisms $G$ of order strictly greater than $g.$ If the signature of the action is $(h; m_1, \ldots, m_l),$ then the Riemann-Hurwitz formula ensures that $$2(g-1)>g(2h-2+l-\Sigma_{j=1}^l\tfrac{1}{m_j})$$and, after straightforward computations, one can see that necessarily $h=0$ and $l=7.$ Thus, \begin{equation} \label{aa}\Sigma_{j=1}^7 \tfrac{1}{m_j} > 3 + \tfrac{2}{g}\end{equation}If $v$ is the number of periods $m_j$ that are different from 2 then \eqref{aa} implies that  $v \in \{0,1,2\}.$  

\s

If $v=0$ then the signature of the action is $(0; 2, \stackrel{7}{\ldots},2)$ and the order of $G$ is $\tfrac{4}{3}(g-1).$ However, as $g-1$ is assumed to be prime, we obtain that $g=4,$ and this contradicts the assumption that the order of $G$ is strictly greater than the genus.

\s

If $v=1$ then the signature of the action is $(0; 2, \stackrel{6}{\ldots}, 2, a)$ for some $a \geqslant 3$ which satisfies, by \eqref{aa}, the inequality $2a < g.$ Note that the order of $G$ is $\tfrac{2a}{2a-1}(g-1),$ but, as $g-1$ is assumed to be prime, we see that necessarily $2a=g;$ a contradiction.

\s

Finally, if $v=2$  then the signature of the action is $(0; 2, \stackrel{5}{\ldots}, 2, a, b)$ for some $a,b \geqslant 3$ that, by \eqref{aa}, satisfy $\tfrac{1}{a}+\tfrac{1}{b} > \tfrac{1}{2}.$ It follows that  the signature of the action is either $$(0; 2, \stackrel{5}{\ldots}, 2, 3,3),\,\, (0; 2, \stackrel{5}{\ldots}, 2, 3, 4)\,\, \mbox{ or }\,\, (0; 2, \stackrel{5}{\ldots}, 2, 3, 5)$$ and, consequently, the order of $G$ is either $\tfrac{12}{11}(g-1),  \,\, \tfrac{24}{23}(g-1) \,\, \mbox{ or }\,\,\tfrac{60}{59}(g-1).$

Note that, as before, the assumption that $g-1$ is prime, implies that $g$ equals 12, 24 or 60 respectively. The contradiction is obtained after noticing that, in every case, the order of $G$ agrees with the genus.
\end{proof}

\begin{lemma} \label{L22} For each even integer $g \geqslant 4,$ there is a complex four-dimensional family of compact Riemann surfaces $S$ of genus $g$ with a group of automorphisms $G$ isomorphic to the dihedral group of order $g$ such that the signature of the action of $G$ on $S$ is $(0; 2, \stackrel{6}{\ldots}, 2, \frac{g}{2})$.
\end{lemma}

\begin{proof} Let $\Delta$ be a Fuchsian group of signature $(0; 2, \stackrel{6}{\ldots}, 2, \frac{g}{2})$ with canonical presentation $$\Delta=\langle x_1, \ldots, x_7 : x_1^2 = \cdots =x_6^2= x_7^{\frac{g}{2}}=x_1\cdots x_7=1\rangle,$$and consider the dihedral group $\mathbf{D}_{\frac{g}{2}}=\langle r, s : r^{\frac{g}{2}}=s^2=(sr)^2=1\rangle$ of order $g.$ Note that  \begin{equation*} \label{action2}\Delta \to \mathbf{D}_{\frac{g}{2}} \, \mbox{ given by } \, x_1, \ldots, x_5 \mapsto s, x_6 \mapsto sr^{-1} \, \mbox{ and } \, x_7 \mapsto r\end{equation*}is a surface-kernel epimorphism of signature $(0; 2, \stackrel{6}{\ldots}, 2, \frac{g}{2})$. In addition, the equality $$2(g-1)=g[0 - 2 + 6(1-\tfrac{1}{2} ) + (1- \tfrac{2}{g})]$$shows that the Riemann-Hurwitz formula is satisfied for a $g$-fold regular covering map from a Riemann surface of genus $g$ onto the projective line with six branch values marked with 2 and with one branch value marked with $\tfrac{g}{2}$. The existence of the  family follows from Riemann's existence theorem.
\end{proof}

{\bf Notation.} From now on, we shall denote the family of all those  surfaces $S$ of genus $g \geqslant 4$ with a group of automorphisms $G$ isomorphic to the dihedral group of order $g$ such that the signature of the action of $G$ on $S$ is $(0; 2, \stackrel{6}{\ldots}, 2, \frac{g}{2})$ by $\mathscr{V}_g.$

\begin{lemma} \label{L31} Let $g \geqslant 3$ be an odd integer. If $g-1$ is a power of two then 
there are no complex four-dimensional families of  Riemann surfaces of genus $g$ with strictly more than $g-1$ automorphisms.
\end{lemma}

\begin{proof} Assume the existence of a complex four-dimensional family of  Riemann surfaces of genus $g$ with a group of automorphisms $G$ of order strictly greater than $g-1,$ and denote the signature of the action by $(h; m_1, \ldots, m_l).$ Then the Riemann-Hurwitz formula ensures that $$4>2h+l-\Sigma_{j=1}^l\tfrac{1}{m_j};$$showing that $h=0$ and $l=7,$ and consequently $\label{ee1} 3 < \Sigma_{j=1}^7 \tfrac{1}{m_j} \leqslant \tfrac{7}{2}.$ By proceeding analogously as done in the proof of Lemma \ref{L21} one sees that the signature is either $$(0; 2, \stackrel{7}{\ldots}, 2), \,(0; 2, \stackrel{6}{\ldots}, 2, a), \,(0; 2, \stackrel{5}{\ldots}, 2, 3,3), \,(0; 2, \stackrel{5}{\ldots}, 2, 3,4) \mbox{ or } (0; 2, \stackrel{5}{\ldots}, 2, 3,5)$$for some $a \geqslant  3.$ It follows that the order of $G$ is either $ \tfrac{4}{3}(g-1), \, \tfrac{2a}{2a-1}(g-1), \, \tfrac{12}{11}(g-1), \, \tfrac{24}{23}(g-1)$ or $\tfrac{60}{59}(g-1).$ The contradiction is obtained after noticing that if $g-1$ is a power of $2,$ then the aforementioned fractions are not integers. 
\end{proof}

\begin{lemma}\label{L32} Let $g \geqslant 3$ be an odd integer. There are:
\begin{enumerate}
\item[(a)] a complex four-dimensional family  of compact Riemann surfaces $S$ of genus $g$ with a group of automorphisms $G$ isomorphic to the cyclic group of order $g-1$ such that the signature of the action of $G$ on $S$ is $(1; 2, \stackrel{4}{\ldots}, 2)$, and 
\item[(b)] a complex four-dimensional family of compact Riemann surfaces $S$ of genus $g$ with a group of automorphisms $G$ isomorphic to the dihedral group of order $g-1$ such that the signature of the action of $G$ on $S$ is $(1; 2, \stackrel{4}{\ldots}, 2)$.
\end{enumerate}
\end{lemma}

\begin{proof} Let $\Delta$ be a Fuchsian group of signature $(1; 2, \stackrel{4}{\ldots}, 2)$ with canonical presentation $$\Delta=\langle \alpha_1, \beta_1, x_1, x_2, x_3, x_4 : x_1^2 =x_2^2= x_3^2=x_4^2=\alpha_1 \beta_1 \alpha_1^{-1}\beta_1^{-1}x_1x_2x_3x_4=1\rangle,$$and consider the cyclic group $C_{g-1}=\langle t : t^{g-1}=1\rangle$ and the dihedral group $\mathbf{D}_{\frac{g-1}{2}}=\langle r,s : r^{\frac{g-1}{2}}=s^2=(sr)^2=1 \rangle.$  The homomorphisms \begin{equation} \label{aa1}\Delta \to C_{g-1} \, \mbox{ given by } \, \alpha_1 \mapsto t, \, \beta_1 \mapsto 1, \, x_1, x_2, x_3, x_4 \mapsto t^{\frac{g-1}{2}}\end{equation}\begin{equation*} \label{aa2}\Delta \to \mathbf{D}_{\frac{g-1}{2}} \, \mbox{ given by } \, \alpha_1, \beta_1 \mapsto 1, \, x_1, x_2 \mapsto s, \,  x_3, x_4 \mapsto sr\end{equation*}are surface-kernel epimorphisms of signature $(1; 2, \stackrel{4}{\ldots}, 2)$. In addition, the equality $$2(g-1)=(g-1)[2 - 2 + 4(1-\tfrac{1}{2} )]$$shows that the Riemann-Hurwitz formula is satisfied for a $(g-1$)-fold regular covering map from a Riemann surface of genus $g$ onto a Riemann surface of genus 1 with four branch values marked with 2. 

Thus, the existence of the desired families follows from Riemann's existence theorem.
\end{proof}

{\bf Notation.} From now on, we shall denote the family  of all those  surfaces $S$ of genus $g \geqslant 3$ with a group of automorphisms $G$ isomorphic to the cyclic group of order $g-1$ (to the dihedral group of order $g-1$) such that the signature of the action of $G$ on $S$ is $(1; 2, \stackrel{4}{\ldots}, 2)$ by $\mathscr{U}_g^1$ (by $\mathscr{U}_g^2$).

\begin{theo} \label{t4}
If $B=\{g \in \mathbb{N}: g \geqslant 3 \}$ then $N_4(g,B)$ does not exist.
\end{theo}
\begin{proof}
We shall proceed by contradiction. Let us assume that $N_{4}(g,B)$ exists and that $$N_{4}(g,B)=ag + b \,\, \mbox{ for suitable (and fixed) } \,\, a,b \in \mathbb{Z}.$$

We claim that $a=1.$ Indeed:
\begin{enumerate}
\item Clearly $a$ cannot be zero (consider Lemma \ref{L22} with $g=b+1$).
\item If $a$ were negative (and therefore $b$ must be positive) then for each $g \geqslant -\tfrac{b}{a}+1$ the number $N_{4}(g,B)$ would be negative; a contradiction. 
\item If $a$ were strictly greater than 1 then for \begin{displaymath} g >  \left\{ \begin{array}{ll}
 \,\,\,2 & \textrm{if $b \geqslant 0$}\\
\tfrac{-b}{a-1} & \textrm{if $b < 0$}
  \end{array} \right.
\end{displaymath} the number $N_{4}(g,B)$ would exceed $g;$ this fact contradicts Lemmata \ref{L21} and \ref{L31}.
\end{enumerate}

Furthermore, by Lemma \ref{L31}, we see that necessarily $b \leqslant -1.$

\s

It follows that for every $g \geqslant 2,$ there is a complex four-dimensional family of  Riemann surfaces $S$ of genus $g$ with a group of automorphisms $G$ of order $g+b.$  If the signature of the action of $G$ on $S$ is $(h; m_1, \ldots, m_l)$ then each period $m_j$ must equal $g+b,$ since otherwise $g \equiv -b \mbox{ mod } m_j$  for some $m_j,$ contradicting the fact that the family exists for all $g\geqslant 2.$ In particular, we obtain that $G$ is necessarily isomorphic to the cyclic group and the Riemann-Hurwitz formula implies that $$2(g-1)=(g+b)[2h-2+l(1-\tfrac{1}{g+b})].$$Hence $b=1-\tfrac{3}{5}g$ if $h=0,$ $b=\tfrac{1}{2}(1-g)$ if $h=1$ and $b=\tfrac{-1}{3}(g+1)$ if $h=2,$ showing that the existence of the family fails to be true for all genus.
\end{proof}
%
%

Once the non-existence of $N_{4}(g,B)$ has been proved, it makes sense to state the following theorem.

\begin{theo} Let $A_1=\{ g \in \mathbb{N}: g \geqslant 3 \mbox{ is odd}\}$ and $A_2=\{ g \in \mathbb{N}: g \geqslant 4 \mbox{ is even}\}.$ Then
$$N_4(g, A_1)=g-1 \,\,\, \mbox{ and } \,\,\,\, N_4(g, A_2)=g.$$
\end{theo}

\begin{proof}
The proof follows directly from Lemmata \ref{L21}, \ref{L22}, \ref{L31} and \ref{L32}.
\end{proof}

\begin{rema}
It is worth observing that the phrase {\it for at least one} $g \in A_j$ in the second statement of the definition of $N_4(g,A_j)$ is not vacuous. Indeed, it is not a difficult task to verify the following facts.
\s

\begin{enumerate}
\item For each $g \geqslant 7$ such that $g \equiv 3 \mbox{ mod } 4$ there exists a complex four-dimensional family of  Riemann surfaces of genus $g$ with a group of automorphisms isomorphic to the dihedral group of order  $g+1$ such that the signature of the action is $(0; 2, \stackrel{6}{\ldots},2, \tfrac{g+1}{4})$.

\s

\item For each $g \geqslant 4$ such that $g \equiv 4 \mbox{ mod } 6$ there exists a complex four-dimensional family of  Riemann surfaces of genus $g$ with a group of automorphisms isomorphic to the dihedral group of order  $\tfrac{4}{3}(g-1)$ such that the signature of the action is $(0; 2, \stackrel{7}{\ldots},2)$.
\end{enumerate}
\end{rema}

\section{The family $\mathscr{V}_g$} \label{s6}

\begin{prop} \label{L23} Let $g \geqslant 4$ be an even integer. 
If $\tfrac{g}{2}$ is a prime number then $\mathscr{V}_g$ is the unique complex four-dimensional family of compact Riemann surfaces of genus $g$ with $g$ automorphisms.
\end{prop}

\begin{proof} Let $\mathscr{V}$ be a complex four-dimensional family of  Riemann surfaces $S$ of genus $g$ with a group of  automorphisms $G$ of order $g$ acting with signature $(h; m_1, \ldots, m_l).$  As argued in the proof of Lemma \ref{L21}, we observe that $h=0$ and $l=7$ and therefore \begin{equation} \label{uu2}
\Sigma_{j=1}^7 \tfrac{1}{m_j}=3+\tfrac{2}{g}.
\end{equation}

We denote  the number of periods $m_j$ that are different from 2 by $v$. Clearly,  $v=0$ if and only if $g = 4.$ We now assume $q=\tfrac{g}{2}$ to be prime and notice that this fact implies that if some $m_j$ is different from 2 then $m_j \geqslant q.$   We claim that $v = 1$ provided that $g \geqslant 6.$ Indeed, if $v \geqslant 2$ then \eqref{uu2} implies that $$3+\tfrac{1}{q}  \leqslant \tfrac{v}{q} +\tfrac{7-v}{2} \iff \tfrac{v-1}{2} \leqslant \tfrac{v-1}{q}$$and then $g = 4.$ Thus, the only possible signature of the action of $G$ on $S$ is $(0; 2, \stackrel{6}{\ldots},2, q).$ 

If $S$ does not belong to  $\mathscr{V}_g$ then $G \cong 2q.$ However, this situation is impossible since there are no surjective homomorphisms from a Fuchsian group of signature $(0; 2, \stackrel{6}{\ldots},2, q)$ onto $C_{2q}$; thus $\mathscr{V}=\mathscr{V}_{g}.$
\end{proof}

\begin{prop} \label{L24} Let $g \geqslant 6$ be an even integer such that $\tfrac{g}{2}$ is prime.  Then then family $\mathscr{V}_g$ consists of at most $\frac{g+2}{4}$ equisymmetric strata.

\end{prop}

\begin{proof} Set $g \geqslant 6$ such that $q=\tfrac{g}{2}$ is prime. Let $\theta: \Delta \to \mathbf{D}_{q}=\langle r, s : r^q=s^2=(sr)^2=1\rangle$ be a surface-kernel epimorphism representing an action of $G$ on $S \in \mathscr{V}_g,$ with $\Delta$ canonically presented as in the proof of  Lemma \ref{L22}. Similarly as done in the proof of Proposition \ref{L14} we shall introduce some notation. We write $m \in \{1, \ldots, q-1\}$ and $n_j \in \{0, \ldots, q-1\} $ for $j=1, \ldots, 6$  such that $$ sr^{n_j}=\theta(x_j) \,\, \mbox{ for } \,\, j=1, \ldots, 6 \,\, \mbox{ and } \,\, r^m=\theta(x_7).$$If $n_j \neq 0$ then we shall denote by $m_j$ its  inverse in the field of $q$ elements. The automorphism of $\mathbf{D}_{q}$ given by $(r,s) \mapsto (r^{\alpha}, sr^{\beta})$ is denoted by $\phi_{\alpha,\beta},$ for
$1 \leqslant \alpha \leqslant q-1$ and $0 \leqslant \beta \leqslant q-1.$ We also restate two  claims which were  proved in the proof of Proposition \ref{L14}.

\s

{\bf Claim 1.} If $n_j=1$ and $n_k=0$ for $k < j$ then we can assume $n_{j+1}=0$ or $n_{j+1}=1.$

\s

{\bf Claim 2.} Up to equivalence, we can assume $n_1=n_2=0.$

\s

We shall proceed by studying separately the cases $n_3=0$ and $n_3 \neq 0.$

\s

{\it Type 1.} Suppose $n_3=0.$ 

\s

Assume $n_4=0.$ If $n_5 \neq 0$ then we consider the automorphism $\phi_{m_5,0}$  to notice that, up to equivalence, $n_5=1.$ Thus, $\theta$ is equivalent to either $(s,s,s,s,s,sr^{u},r^{-u})$ or $(s,s,s,s,sr,sr^v, r^{1-v})$ where $u \neq 0$ and $v\neq 1,$ according to $n_5=0$ or $n_5=1$. Note that in the first case, as $u \neq 0,$ the epimorphism  is equivalent to the one in which $u=1;$ namely, equivalent to \begin{equation} \label{epp11}(s,s,s,s,s,sr,r^{-1})\end{equation}Meanwhile, in the latter case, by Claim 2, the epimorphism is equivalent to \begin{equation} \label{epp3}(s,s,s,s,sr,s,r) \,\, \mbox{ and, in turn, equivalent to }\,\,  (s,s,s,s,s,sr^{-1},r).\end{equation}

Now, the transformation $\phi_{-1, 0}\circ \Phi_5$ provides an equivalence between \eqref{epp11} and \eqref{epp3}.

\s

Assume $n_4 \neq 0.$ We then consider the automorphism $\phi_{m_4,0}$   to notice that, up to equivalence, $n_4=1$ and, consequently, by Claim 2, we have that $n_5=0$ or $n_5=1.$ Thereby, $\theta$ is equivalent to either $$(s,s,s,sr,s,sr^u, r^{-1-u}) \,\, \mbox{ or } \,\, \theta_v=(s,s,s,sr,sr,sr^v, r^{-v})$$where $u \neq -1$ and $v \neq 0.$ The first case is equivalent to \eqref{epp3}; indeed, we can consider the transformation $\phi_{-1,0} \circ \Phi_4$ to see that $\theta$ is equivalent to $(s,s,s,s,sr,sr^{-u}, r^{1+u})$ and therefore, by Claim 2, we can assume $u=0.$ For the second case, consider $\Phi_6 \circ \Phi_6$ to notice that $\theta_v$ and $\theta_{-v}$ are equivalent. It follows that there are  at most $\tfrac{q-1}{2}$ pairwise non-equivalent actions given by \begin{equation} \label{eppvarios}(s,s,s,sr,sr,sr^v, r^{-v}) \,\, \mbox{ for some }\,\, v \in \{1, \ldots, \tfrac{q-1}{2}\}.\end{equation}

\s

{\it Type 2.} Suppose $n_3 \neq 0.$ As before, consider the automorphism $\phi_{m_3,0}$  to assume $n_3=1.$ Moreover, again by Claim 2, we see that, up to equivalence, $n_4=0$ or $n_4=1.$ However, we only need to consider the case $n_4=1$ due to the fact that, if $n_4=0$ then the transformation $\phi_{-1,0} \circ \Phi_3$ provides an equivalence between $\theta$ and either \eqref{epp3} or some \eqref{eppvarios}. Thus, we assume that $n_4=1.$ If $n_5=0$ then the transformation $\Phi_4 \circ \Phi_5 \circ \phi_{-1,0}$ shows that the epimorphism is equivalent to either \eqref{epp3} or some \eqref{eppvarios}. Then, we can assume $n_5 \neq 0$ and therefore the epimorphism is equivalent to one of the form \begin{equation*} \theta_{u,v}= (s,s,sr,sr,sr^{u}, sr^{u-v}, r^v)  \end{equation*}where $u,v \in \{1, \ldots, q-1\}.$ Note that the powers of $\Phi_5$ provide the equivalences $\theta_{u,v} \cong \theta_{u-\lambda v,v}$ where $\lambda \in \{1, \ldots, q-1\} .$ We  choose $\lambda=\tfrac{u-v}{v}$ to conclude that $\theta$ is equivalent to $$\theta_{v,v}=(s,s,sr,sr,sr^{v}, s, r^v) \,\, \mbox{ for some }\,\, v \in \{1, \ldots, q-1\}.$$ Now, consider $\phi_{-1,0} \circ \Phi_3 \circ \Phi_4 \circ \Phi_5$ to conclude that $\theta_{v,v}$ is equivalent to \eqref{eppvarios}.
\s

All the above says that $\theta$ is equivalent to either \eqref{epp11} or some \eqref{eppvarios}. Hence $\mathscr{V}_g$ consists of at most $\tfrac{q-1}{2}+1=\tfrac{g+2}{4}$ equisymmetric strata, as claimed.
\end{proof}

\begin{prop} \label{celular}
 Let $g \geqslant 6$ be an even integer such that $\tfrac{g}{2}$ is odd. If $S \in \mathscr{V}_g$ then the Jacobian variety $JS$ decomposes, up to isogeny, as$$JS\sim A\times \Pi_{d\in \Omega(\frac{g}{2})} B_d^2, $$
where $A$ is an abelian surface  and $B_d$ is an abelian variety of dimension $\varphi(\frac{g}{2d}).$ Moreover $$A \sim JS_{\langle r \rangle}\,\, \mbox{ and } \,\, \Pi_{d \in \Omega(\frac{g}{2})}B_d \sim JS_{\langle s \rangle}$$and therefore $JS\sim JS_{\langle r \rangle}\times JS_{\langle s \rangle}^2. $

\end{prop}
\begin{proof}

Let $S \in \mathscr{V}_g$ with $g \geqslant 6$ and $n=\tfrac{g}{2}$ odd. As noticed in the proof of Proposition \ref{desco3dim} and keeping the same notations as in there, the non-trivial rational  irreducible representations of $\mathbf{D}_n$ are $\chi_2$ and $W_d$ with $d \in \Omega(n)$ and therefore the group algebra decomposition of each $JS$ with respect to $G$ is given by \begin{equation} \label{rojo1}JS \sim B_2 \times \Pi_{d \in \Omega(n)}B_d^2.\end{equation}

The fact that  the involutions of $\mathbf{D}_n$ are conjugate implies that the  dimension of  $B_2$ and $B_d$ in \eqref{rojo1}  does not depend on the stratum to which $S$ belongs. So, we assume the action of $G$ on $S$ to be represented by  $(s,s,s,s,s,rs,r).$  Consider the equation \eqref{dimensiones1} to see that $\dim B_2 =2$ and $\dim B_d =\varphi(\tfrac{n}{d}).$ In addition, we consider the induced isogeny \eqref{decoind1} with $H= \langle r \rangle$ and $H=\langle s \rangle$ to see that $$B_2 \sim JS_{\langle r \rangle}\,\, \mbox{ and } \,\, \Pi_{d \in \Omega(n)}B_d \sim JS_{\langle s \rangle}$$respectively, and therefore the proof follows after setting $A=B_2.$
\end{proof}

\begin{rema}
We end this section by remarking two facts concerning the family $\mathscr{V}_g.$
\begin{enumerate}
\item The behavior for $g=4$ is completely different. Indeed, as noticed in \cite{CIg4} (see also \cite{bciracsam})  the family $\mathscr{V}_4$ consists of two  strata, represented by $\theta_1=(r,r,r,r,r,s,sr)$ and $\theta_2=(r,r,r,s,s,s,sr).$ By proceeding analogously as done in the proof of Proposition \ref{celular}, one sees that if $S \in \mathscr{V}_4$ then:

\s

\begin{enumerate}
\item if $S$ belongs to the  stratum defined by $\theta_1$ then  $JS\sim A_1 \times A_2,$ where $A_1 \sim JS_{\langle s \rangle}$ and $A_2 \sim JS_{\langle sr \rangle}$ are abelian surfaces, and
\s

\item if $S$ belongs to the  stratum defined by $\theta_2$ then
$JS\sim E_1\times E_2\times A,$ where $E_1 \sim JS_{\langle r \rangle}$ and $E_2 \sim JS_{\langle s \rangle}$ are elliptic curves and $A \sim JS_{\langle sr \rangle}$ is an abelian surface. 
\end{enumerate}

\s

\item As the reader could expect, if $n=\tfrac{g}{2}$ is even then Propositions \ref{L24} and \ref{celular} are not longer true. For instance, the  stratum defined by  $\eta=(r^{\frac{n}{2}}, r^{\frac{n}{2}}, s,s,s,rs,r)$ is not equivalent to any of the actions determined in Proposition \ref{L24}. Furthermore, if $S$ belongs to the stratum defined by $\eta$ then, by proceeding  as in the proof of Proposition \ref{celular}, one sees that if $n \equiv 0 \mbox{ mod } 4$ then  $JS$ contains two elliptic curves, and if $n \equiv 2 \mbox{ mod } 4$ then  $JS$ contains two elliptic curves and an abelian surface.

\end{enumerate}
\end{rema}

\section{The families $\mathscr{U}^1_g$ and $\mathscr{U}^2_g$} \label{s7}

\begin{prop} \label{L33} Let $g \geqslant 11$ be an odd integer such that  $g-1$ is twice a prime number. Then $\mathscr{U}_g^1$ and $\mathscr{U}_g^2$ are the unique complex four-dimensional families with $g-1$ automorphisms.
\end{prop}

\begin{proof} Let $g \geqslant 11$ be an odd integer and write $g-1=2q$ where $q \geqslant 5$ is a prime number. As the cyclic and dihedral group are the unique groups of order $2q,$ we only need to verify that $(1; 2, \stackrel{4}{\ldots}, 2)$ is the only possible signature for the action of a group $G$ of order $2q$ on a complex-four dimensional family of  Riemann surfaces of genus $1+2q$.

A short computation shows that if signature of the action is not $(1; 2, \stackrel{4}{\ldots}, 2)$ then it is $(h; m_1, \ldots, m_l)$ where either $(h,l)=(2,1)$ or $(h,l)=(0,7).$ It is straightforward to see that the former case is impossible. So, we assume $(h,l)=(0,7)$ and then $\Sigma_{j=1}^7 \tfrac{1}{m_j}=3.$  As argued in the proof of Lemma \ref{L21} and \ref{L31}, one sees that the number $v$ of periods $m_j$ that are different from 2 are either two of three. 

\begin{enumerate}
\item If $v=2$ then the signature of the action is $(0; 2, \stackrel{5}{\ldots},2,  a, b)$ where $a,b \geqslant 3$ satisfy $$\tfrac{1}{a}+\tfrac{1}{b}=\tfrac{1}{2} \,\, \mbox{ and therefore } \,\, a=b=4 \, \mbox{ or } \, a=3, b=6.$$
\item If $v=3$ then the signature of the action is $(0;  2, \stackrel{4}{\ldots},2,  a,b,c)$ where $a,b,c \geqslant 3$  satisfies  $$\tfrac{1}{a}+\tfrac{1}{b}+\tfrac{1}{c}=1 \,\, \mbox{ and therefore }\,\, a=b=c=3.$$
\end{enumerate}

It follows that the order of the group is divisible by $3,4$ or 6. Thus, $q=2$ or $3$ and therefore the genus equals $g=5$ or $g=7$; a contradiction. 
\end{proof}

\begin{rema} The exceptional signatures appearing in the proof of the  proposition above are realized for the unconsidered cases $g=5$ and 7 (see, for example, \cite[Lemma 8]{bciracsam} for $g=5$).

%
%
\end{rema}

\begin{prop} \label{L34} Let $g \geqslant 3$ be an odd integer. 
The family $\mathscr{U}_g^1$ is equisymmetric.
\end{prop}

\begin{proof} For $g=3$ we refer to \cite[Table 5, 3.b]{Brou}.
Assume $g \geqslant 5.$
Let $\Delta$ be a Fuchsian group of signature $(1; 2, \stackrel{4}{\ldots}, 2)$ canonically presented as in the proof of Lemma \ref{L32} and let $\theta: \Delta \to G=\langle t : t^{g-1}=1 \rangle$ be a surface-kernel epimorphism representing an action of $G$ on $S \in \mathscr{U}_g^1.$  We have to prove that $\theta$ is equivalent to the surface-kernel epimorphism \eqref{aa1} in the proof of Lemma \ref{L32}. Clearly $$\theta(x_j)=t^{(g-1)/2} \,\, \mbox{ for } \,\, j=1,2,3,4. $$If we write $\theta(\alpha_1)=t^u$ and $\theta(\beta_1)=t^v$ then as, $\theta$ is surjective, without loss of generality, we can assume $u=1.$ Consider the transformation $A_{1,-v}$ (see \S\ref{braid}) to see that $\theta$ is equivalent to \eqref{aa1}, as desired. 
\end{proof}
%
For $n \geqslant 2$ even, let  $\Lambda(n)=\{ 1 \leqslant d < \tfrac{n}{2}: d \text{ divides } n  \text{ and }  \tfrac{dn}{2} \not\equiv 0 \mbox{ mod } n\}.$

\begin{prop}
Let $g\geqslant 3$ be an odd integer. If $S \in \mathscr{U}_g^1$ then the Jacobian variety $JS$ decomposes, up to isogeny, as follows.
\begin{enumerate}
\s

\item If $\tfrac{g-1}{2}$ is even then $$JS\sim E \times \Pi_{d\in \Lambda(g-1)}B_d$$where $E$ is an elliptic curve isogenous to $JS_G$ and $B_d$ is an abelian variety of dimension $2\varphi(\tfrac{g-1}{d}).$

\item If $\tfrac{g-1}{2}$ is odd then $$JS\sim E \times A \times  \Pi_{d\in \Lambda(g-1)}B_d$$where $A$ is an abelian surface and $E$ and $B_d$ are as before. 
\end{enumerate}
\end{prop}

\begin{proof} Set $n=g-1$ and 
let $\omega$ be a primitive $n$-th  root of unity. For each $0 \leqslant j \leqslant n-1$, we denote by $\chi_j$ the complex irreducible representation of $G = \langle t : t^n =1 \rangle =C_n$ defined as $\chi_{j} : t \mapsto \omega^j.$  After a routine computation, one sees that the collection $\{\chi_d\}$  where $1 \leqslant d \leqslant \tfrac{n}{2}$ and $d$   divides $n$ yields a maximal collection of non-trivial rational irreducible representations of $G$, up to equivalence.

Let $B_d$ denote the factor associated to $\chi_d$ in the group algebra decomposition of $JS$ with respect to $G.$ Clearly, $B_0$ is an elliptic curve isogenous to $JS_G.$ In addition, we observe that $$\dim B_d = \varphi(\tfrac{n}{d})\tfrac{1}{2}\cdot 4 (1- \chi_d^{\langle t^{\frac{n}{2}}\rangle})$$and therefore $B_d = 0$ if and only if  $ \chi_j(t^{\frac{n}{2}})=\omega^{\frac{nd}{2}}=1$ or, equivalently $\tfrac{nd}{2} \equiv 0 \mbox{ mod }n.$ Hence, the group algebra decomposition of $JS$ with respecto to $G$ is $$JS \sim JS_G \times B_{\frac{n}{2}} \times \Pi_{d\in \Lambda(n)}B_d$$where, for each $d\in \Lambda(n),$ the dimension of $B_d$ is $2\varphi(\frac{n}{d}).$ Finally, as $\chi_{\frac{n}{2}}(t)=-1$ we see that  \begin{displaymath} \dim B_{\frac{n}{2}}=\tfrac{1}{2} \cdot 4 (1- \chi_{\frac{n}{2}}^{\langle  t^{\frac{n}{2}}  \rangle}   )= \left\{ \begin{array}{ll}
2 & \textrm{if $\tfrac{n}{2}$ is odd}\\
0 & \textrm{if $\tfrac{n}{2}$ is even}
  \end{array} \right.
\end{displaymath}and  the proof follows after setting $E=B_0$ and $A=B_{\frac{n}{2}}$ when $\tfrac{n}{2}$ is odd.
\end{proof}

\begin{prop} \label{L38} Let $g \geqslant 5$ be an odd integer such that $\tfrac{g-1}{2}$ is a prime number. Then the family $\mathscr{U}_g^2$ consists of at most two equisymmetric strata.
\end{prop}

\begin{proof} Set $q=\tfrac{g-1}{2}$ and assume $q$ to be prime. Let $\Delta$ be a Fuchsian group of signature $(1; 2, \stackrel{4}{\ldots}, 2)$  canonically presented as in Lemma \ref{L32} and let $\theta: \Delta \to G= \mathbf{D}_{q}=\langle r,s : r^{q}=s^2=(sr)^2=1 \rangle$ be a surface-kernel epimorphism representing an action of $G$ on $S \in \mathscr{U}_g^2.$ We write $a=\theta(\alpha_1), b=\theta(\beta_1)$ and $sr^{n_i}=\theta(x_i)$ where $n_i \in \{0, \ldots, q-1\}$ for $i=1,2,3,4$ and  identify $\theta$ with  $(a,b; sr^{n_1}, sr^{n_2}, sr^{n_3}, sr^{n_4}).$

\s

{\bf Claim.} Up to equivalence, we can assume $(a,b)=(1,1)$ or $(a,b)=(1,r).$ 

\s

There are four cases to consider; namely $(a,b)$ equals to either
$$(r^u, r^v), \,\, (sr^{u}, r^v), \,\, (r^u, sr^v) \,\, \mbox{ or }\,\, (sr^u, sr^{v}) \,\,\,\, \mbox{ for some} \,\, u,v \in \{0, \ldots, q-1\}.$$

First of all, note the third and fourth cases can be disregarded, since  $A_{1,1} \circ A_{2,1}$ and $A_{1,1}$ respectively (see \S\ref{braid}), transform them into the second case. 

\s

Assume that $a=r^u$ and $b=r^{v}.$ 
\begin{enumerate}  
\item If $u = 0$ then, up to an automorphism, we can assume $(a,b)=(1,1)$ or $(1,r).$ 
\item If $u \neq 0$ and $\tilde{u}$ is its inverse in the field of $q$ elements, then the transformation $\phi_{-\tilde{u}, 1} \circ A_{2,1} \circ  A_{1,-1} \circ A_{1,-v\tilde{u}}$ allows us to assume that, up to equivalence, $(a,b)=(1,r).$ 
\end{enumerate}

 Assume that $a=sr^u$ and $b=r^{v}$ 
 \begin{enumerate}
 \item If $v=0$ then, up to an automorphism, we can assume $(a,b)=(s,1).$
 \item If $v \neq 0$ and $\hat{v}$ is its inverse in the field of $q$ elements, then $\phi_{\hat{v},0} \circ A_{2,-u\hat{v}}$ shows that, up to equivalence, we can assume $(a,b)=(s, r).$
\end{enumerate}
The proof of the claim follows after noticing that the cases $(s,r)$ and $(1,s)$ are equivalent to the first case under the action of the transformations $C_{1,4}$ and $C_{2,4}$ respectively.

\s

If $(a,b)=(1,1)$ then we can assume $n_1=0$ and $n_2=1.$ Thus, $\theta$ is equivalent to $(1,1; s,sr,sr^{n_3}, sr^{n_3-1}).$ Now, we apply transformation $\Phi_3$ to see that $\theta$ is equivalent to \begin{equation} \label{virus}(1,1; s,sr,sr,s) \,\, \mbox{ and therefore equivalent to }\,\, (1,1; s,s,sr,sr).\end{equation}

If $(a,b)=(1,r)$ then we can assume $n_1=0$ and therefore $\theta$ is equivalent to $(1,r; s, sr^{n_2}, sr^{n_3}, sr^{n_4}).$

\begin{enumerate}
\item If $n_2=0$ then $n_3=n_4.$ If follows that $\theta$ is equivalent to $(1,r; s,s,s,s)$ or to $$ \theta_j:= (1,r^j; s,s,sr,sr) \,\, \mbox{ for some } j\in \{1, \ldots, q-1\}.$$The latter case is equivalent to \eqref{virus}, since $C_{2,3} \circ C_{2,2}$ identifies $\theta_j$ with $\theta_{j-1}.$

\s

\item If $n_2 \neq 0$ then $\theta$ is equivalent to $(1,r^j; s, sr, sr^{n_3}, sr^{n_3-1})$  for some $j \in \{1, \ldots, q-1\}.$ Note that $\Phi_3$ shows that  $\theta$ is equivalent  to $(1, r^j; s,sr,sr,s);$ then, we see that $\theta$ is equivalent to $\theta_j.$
\end{enumerate}
All the above says that there are at most 2  strata given by  $\Theta_1=(1,1; s,s,sr,sr)$ and $\Theta_2=(1, r; s,s,s,s).$
\end{proof}

\begin{prop} \label{dedo}  Let $g \geqslant 7$ be an odd integer such that $\tfrac{g-1}{2}$ is odd. If $S \in \mathscr{U}_g^2$ then the Jacobian variety $JS$ decomposes, up to isogeny, as$$JS\sim E \times  A\times \Pi_{d\in \Omega(\frac{g-1}{2})}B_d^2,$$
where $A$ is an abelian surface,  $B_d$ is an abelian variety of dimension $\varphi(\frac{g-1}{2d})$ and $E$ is an elliptic curve isogenous to $JS_G$. Furthermore $$\mbox{Prym}(S_{\langle r \rangle} \to S_G) \sim A \,\, \mbox{ and } \,\, \mbox{Prym}(S_{\langle s \rangle} \to S_G) \sim \Pi_{d\in \Omega(\frac{g-1}{2})}B_d$$and therefore $JS\sim JS_G \times \mbox{Prym}(S_{\langle r\rangle} \to S_G) \times \mbox{Prym}(S_{\langle s\rangle} \to S_G)^2.$

\end{prop}

\begin{proof}
 Set $n=\tfrac{g-1}{2}$ and assume that $n$ is odd. Similarly as noticed in the proof of Proposition \ref{desco3dim}(a), the dimension of the factors arising in the group algebra decomposition of $JS$ does not depend on the  stratum to which $S$ belongs. So, we assume the action to be represented by $\Theta_2.$ Now, keeping the same notation as before, the rational irreducible representations of $G$ are $\chi_1, \chi_2$ and $\psi_d$ with $d\in \Omega(n)$ and  $$JS \sim B_1 \times B_2 \times \Pi_{d \in \Omega(n)}B_d^2,$$where  $B_1 \sim JS_G$ is an elliptic curve. As $\chi_2^{\langle s \rangle}=0$ and $\psi_d^{\langle s \rangle}=1$ for $d \in  \Omega(n)$ one sees that $\dim B_2=2$ and $\dim B_d=\varphi(\tfrac{n}{d}).$ Now, we consider the induced isogeny \eqref{decoind2} with $H_1=\langle r \rangle$ and $H_2=G$ to see that $$B_2 \sim \mbox{Prym}(S_{\langle r \rangle} \to S_G).$$Similarly,  consider the induced isogeny \eqref{decoind2} with $H_1=\langle s \rangle$ and $H_2=G$ to see that $$\Pi_{d\in \Omega(n)}B_d \sim \mbox{Prym}(S_{\langle s \rangle} \to S_G),$$and the result follows by setting $E=B_1$ and $A=B_2.$
\end{proof}

\begin{rema} The isogeny decomposition of the Jacobian varieties $JS$ for $S \in \mathscr{U}_5^2$ differs from the stated in Proposition \ref{dedo} for the case $g \geqslant 7.$ Furthermore, the decomposition depends on the  stratum to which $S$ belongs (due to the fact that the involutions of $\mathbf{D}_2$ are non-conjugate). Indeed

\begin{enumerate}
\item If $S$ belongs to the  stratum represented by $\Theta_1$ then $JS\sim E_1\times E_2\times E_3\times A,$ 
where $E_1 \sim JS_G, E_2 \sim JS_{\langle s\rangle}$ and $E_3 \sim JS_{\langle sr\rangle}$ are elliptic curves and $A \sim JS_{\langle r\rangle}$ is an abelian surface. 

\s

\item If $S$ belongs to the stratum represented by $\Theta_2$ then $JS\sim E\times A_1\times A_2,$
where $E \sim JS_G$ is an elliptic curve and $A_1 \sim JS_{\langle r\rangle}, A_2 \sim JS_{\langle sr\rangle}$ are abelian surfaces. 
\end{enumerate}

\end{rema}

\s

{\bf Acknowledgments.} The authors are very grateful to {\it SageMath} (www.sagemath.org) and its developers for generously providing a useful software which allows the authors to perform helpful experiments throughout the preparation of this manuscript, towards obtaining general results. The authors are also very grateful to the referee for useful suggestions and comments.


\begin{thebibliography}{9}

\bibitem{Accola}
{\sc R. Accola,} {\em On the number of automorphisms of a closed Riemann surface}, Trans. Am. Math. Soc., {\bf 131} (1968), 398--408.

\bibitem{Ba} 
{\sc P. Barraza and A. M. Rojas,}  {\em The group algebra decomposition of Fermat curves of prime degree}. Arch. Math. (Basel) {\bf 104} (2015), no. 2, 145--155.

\bibitem{bcitesis} 
{\sc G. Bartolini,} 
{\em On the Branch Loci of Moduli Spaces of Riemann Surfaces.}
Link{\"o}ping Studies in Science and Technology Dissertations, {\bf 1440} , Link{\"o}ping 2012.

\bibitem{bciracsam} 
{\sc G. Bartolini, A. F. Costa, and M. Izquierdo,} 
{\em On the orbifold structure of the moduli space of Riemann surfaces of genera four and five.}
Rev. R. Acad. Cienc. Exactas Fis. Nat. Ser. A Mat. RACSAM {\bf 108} (2014), no. 2, 769--793. 


\bibitem{BJ} 
{\sc M. V. Belolipetsky and G. A. Jones,} 
{\em Automorphism groups of Riemann surfaces of genus $p + 1$, where $p$ is prime.}
Glasg. Math. J. {\bf 47} (2005), no. 2, 379--393.
\bibitem{bl}
{\sc Ch. Birkenhake and H. Lange},
 { \em Complex Abelian Varieties,} $2^{nd}$ edition,
Grundl. Math. Wiss. {\bf 302}, Springer, 2004.

%
%
\bibitem{Brou}
{\sc{S. A. Broughton,}}  {\em{Classifying finite groups actions on surfaces of low genus,}} J. Pure Appl. Algebra {\bf 69} (1990), no. 3, 233--270.

\bibitem{b}
{\sc S. A. Broughton},
 { \em The equisymmetric stratification of the moduli space and the Krull dimension of mapping class groups,} Topology Appl. {\bf 37} (1990), no. 2, 101--113.
 
\bibitem{BCI}
{\sc{E. Bujalance, A. F. Costa and M. Izquierdo,}}  {\em{On Riemann surfaces of genus g with 4g automorphisms,}} Topology and its Applications {\bf 218} (2017) 1--18.




\bibitem{CMA}
{\sc A. Carocca, H. Lange and R. E. Rodr\'iguez}, {\em Jacobians with complex multiplication.} Trans. Amer. Math. Soc. {\bf 363} (2011), no. 12, 6159--6175.




\bibitem{CRC}
{\sc A. Carocca and S. Reyes-Carocca}, {\em Riemann surfaces of genus $1+q^2$ with $3q^2$ automorphisms,} Preprint, arXiv: 1911.04310v1.

 \bibitem{d1}
{\sc A. Carocca, S. Recillas and R. E. Rodr\'iguez},
 {\em Dihedral groups acting on Jacobians,} Contemp. Math. {\bf 311} (2011), 41--77.
 
\bibitem{cr}
{\sc A. Carocca and R. E. Rodr\'iguez,}
{\em Jacobians with group actions and rational idempotents.}
J. Algebra \textbf{306} (2006), no. 2, 322--343.

%
%

\bibitem{CIg4}
{\sc A. F. Costa and M. Izquierdo,}
{\em Equisymmetric strata of the singular locus of the moduli space of Riemann surfaces of genus 4,} Geometry of Riemann surfaces, 120--138, London Math. Soc. Lecture Note Ser., 368, Cambridge Univ. Press, Cambridge, 2010. 


\bibitem{CI}
{\sc A. F. Costa and M. Izquierdo,}
{\em One-dimensional families of Riemann surfaces of genus $g$ with $4g + 4$ automorphisms,} Rev. R. Acad. Cienc. Exactas F\'is. Nat. Ser. A Mat. RACSAM {\bf 112} (2018), no. 3, 623--631.

\bibitem {CIY2}
{\sc A. F. Costa, M. Izquierdo, and D. Ying,} 
{\em On cyclic $p$-gonal Riemann surfaces with several $p$-gonal morphisms}. Geom. Dedicata {\bf 147}
(2010), 139-147

\bibitem{Do}
{\sc R. Donagi and E. Markman}, {\em Spectral covers, algebraically completely integrable, Hamiltonian systems, and moduli of bundles}, in: Integrable Systems and Quantum Groups, Montecatini Terme, 1993, in: Lecture Notes in Math., vol. 1620, Springer, Berlin, 1996, pp. 1--119.


\bibitem{Harvey1}
{\sc J. Harvey,}
{\em Cyclic groups of automorphisms of a compact Riemann surface}. 
Quarterly J. Math. {\bf 17}, (1966), 86--97.

\bibitem{Harvey}
{\sc J. Harvey,}
{\em On branch loci in Teichm\"{u}ller space},
Trans. Amer. Math. Soc. {\bf 153} (1971), 387--399.


\bibitem{nos}
{\sc R. A. Hidalgo, L. Jim\'enez, S. Quispe and S. Reyes-Carocca,} {\em Quasiplatonic curves with symmetry group $\mathbb{Z}_2^2 \rtimes \mathbb{Z}_m$ are definable over $\mathbb{Q}$,} Bull. London Math. Soc. {\bf 49} (2017) 165--183.

\bibitem{IJR}
{\sc M. Izquierdo, L. Jim\'enez, A. Rojas,}
{\em Decomposition of Jacobian varieties of curves with dihedral actions via equisymmetric stratification},   Rev. Mat. Iberoam. {\bf 35}, No. 4 (2019), 1259--1279.


\bibitem{IJRC}
{\sc M. Izquierdo, G. A. Jones and S. Reyes-Carocca,}
{\em Groups of automorphisms of Riemann surfaces and maps  of genus $p+1$ where $p$ is prime}. To appear in Ann. Acad. Sci. Fenn. Math., arXiv:2003.05017 

\bibitem{IRC}
{\sc M. Izquierdo and S. Reyes-Carocca,}
{\em A note on large automorphism groups of compact Riemann surfaces}. J. Algebra 547 (2020), 1--21.

\bibitem{KR}
{\sc E. Kani and M. Rosen,} {\em Idempotent relations and factors of Jacobians},
Math. Ann. {\bf 284} (1989), 307--327.
%
%
%

\bibitem{K1}
{\sc R. S. Kulkarni,}
{ \em A note on Wiman and Accola-Maclachlan surfaces}.
Ann. Acad. Sci. Fenn., Ser. A 1 Math. {\bf 16} (1) (1991)
83--94.

%
%
%
%
%
\bibitem{l-r}
{\sc H. Lange and S. Recillas,}
{ \em Abelian varieties with group actions}.
J. Reine Angew. Mathematik, \textbf{575} (2004) 135--155.

\bibitem{LR2}
{\sc H. Lange and S. Recillas,}
{ \em Prym varieties of pairs of coverings}.
Adv. Geom. \textbf{4} (2004) 373--387.

%
\bibitem{Mac}
{\sc C. Maclachlan,} 
{\em A bound for the number of automorphisms of a compact Riemann surface},
J. London Math. Soc. {\bf 44} (1969), 265--272.

%
%



\bibitem{PA}
{\sc J. Paulhus and A. M. Rojas,}
{ \em Completely decomposable Jacobian varieties in new genera},  Experimental Mathematics {\bf 26} (2017), no. 4, 430--445.

\bibitem{d3}
{\sc S. Recillas and R. E. Rodr\'iguez,} {\em Jacobians and representations of $S_3$}, Aportaciones Mat. Investig. {\bf 13}, Soc. Mat. Mexicana, M\'exico, 1998.
%

\bibitem{yonil}
{\sc S. Reyes-Carocca,} {\em Nilpotent groups of automorphisms of families of Riemann surfaces},
Preprint, arXiv:2004.06506 

\bibitem{yo2}
{\sc S. Reyes-Carocca,} {\em On the one-dimensional family of Riemann surfaces of genus $q$ with $4q$ automorphisms,} 
J. of Pure and Appl. Algebra 223, no. 5 (2019), 2123--2144.


\bibitem{yo}
{\sc S. Reyes-Carocca,} {\em On compact Riemann surfaces of genus $g$ with $4g-4$ automorphisms},
Israel J. Math. {\bf 237} (2020), 415--436.
%
%
%
\bibitem{kanirubiyo}
{\sc S. Reyes-Carocca and R. E. Rodr\'iguez,} {\em A generalisation of Kani-Rosen decomposition theorem for Jacobian varieties}, Ann. Sc. Norm. Super. Pisa Cl. Sci. (5) {\bf 19} (2019), no. 2, 705--722. 

\bibitem{RCR}
{\sc S. Reyes-Carocca and R. E. Rodr\'iguez,} {\em On Jacobians with group action and coverings}, Math. Z. (2020) {\bf 294}, 209--227.

\bibitem{Ri}
{\sc J. Ries}, {\em The Prym variety for a cyclic unramified cover of a hyperelliptic curve}, J. Reine Angew. Math. {\bf 340} (1983) 59--69.


\bibitem{yoibero}
{\sc A. M.   Rojas}, {\em Group actions on Jacobian varieties}, Rev. Mat. Iber. {\bf 23} (2007), 397--420.

\bibitem{sage}
{\em SageMath, the Sage Mathematics Software System (Version 9.0)}, The Sage Developers, 2019, www.sagemath.org.




\bibitem{singerman2}
{\sc D. Singerman}, {\em Finitely maximal Fuchsian groups}, J. London Math. Soc. (2)  {\bf 6}, (1972), 29--38.

\bibitem{singerman}
{\sc D. Singerman}, {\em Subgroups of Fuchsian groups and finite permutation groups}, Bull. London Math. Soc.  {\bf 2}, (1970), 319--323.


%
\bibitem{Wi}
{\sc A. Wiman,} {\em \"{U}ber die hyperelliptischen Curven und diejenigen von Geschlechte p - Jwelche eindeutige Tiansformationen in sich zulassen}. Bihang till K. Svenska Vet.-Akad. Handlingar, Stockholm {\bf 21} (1895-6) 1--28.


\end{thebibliography}
\end{document}